\documentclass[12pt]{amsart}
\usepackage{amssymb}
\usepackage{hyperref}
\usepackage{tabularx}
\usepackage{booktabs}
\usepackage{caption}
\usepackage{mathdots}

\newtheorem{thm}{Theorem}[section]
\newtheorem{lem}[thm]{Lemma}
\newtheorem{cor}[thm]{Corollary}
\newtheorem{prop}[thm]{Proposition}
\newtheorem{ex}[thm]{Example}

\newtheorem*{prob*}{Open problem}

\theoremstyle{definition}

\newtheorem{defi}[thm]{Definition}

\theoremstyle{remark}

\newtheorem{rem}[thm]{Remark}
\newtheorem*{rem*}{Remark}

\DeclareMathOperator{\Hom}{Hom}

\newcommand{\kringel}{\mathbin{\raise1pt\hbox{$\scriptstyle\circ$}}}
\newcommand{\pkt}{\mathbin{\raise0pt\hbox{$\scriptstyle\bullet$}}}

\newcommand{\C}{\mathbb{C}}

\newcommand{\K}{\mathbb{K}}

\newcommand{\N}{\mathbb{N}}

\newcommand{\R}{\mathbb{R}}

\newcommand{\ad}{{\rm ad}}

\newcommand{\End}{{\rm End}}
\newcommand{\Der}{{\rm Der}}
\newcommand{\Inn}{{\rm Inn}}

\newcommand{\diag}{\mathop{\rm diag}}

\newcommand{\Lg}{\mathfrak{g}}
\newcommand{\Lh}{\mathfrak{h}}

\newcommand{\CB}{\mathcal{B}}
\newcommand{\CC}{\mathcal{C}}

\newcommand{\abs}[1]{\lvert#1\rvert}

\newcommand{\al}{\alpha}
\newcommand{\be}{\beta}

\newcommand{\ep}{\varepsilon}

\newcommand{\la}{\lambda}

\newcommand{\Ga}{\Gamma}

\newcommand{\ra}{\rightarrow}

\renewcommand{\phi}{\varphi}

\newcommand{\AID}{{\rm AID}}

\newcommand{\bigzero}{\mbox{\normalfont\Large\mdseries \; 0\; }}

\newcommand{\rvline}{\hspace*{-\arraycolsep}\vline\hspace*{-\arraycolsep}}

\begin{document}


\title[Almost inner derivations]{Almost inner derivations of 2-step nilpotent Lie algebras of genus 2}

\author[D. Burde]{Dietrich Burde}
\author[K. Dekimpe]{Karel Dekimpe}
\author[B. Verbeke]{Bert Verbeke}
\address{Fakult\"at f\"ur Mathematik\\
Universit\"at Wien\\
Oskar-Morgenstern-Platz 1\\
1090 Wien \\
Austria}
\email{dietrich.burde@univie.ac.at}
\address{Katholieke Universiteit Leuven Kulak\\
E. Sabbelaan 53 bus 7657\\
8500 Kortrijk\\
Belgium}
\email{karel.dekimpe@kuleuven.be}
\email{bert.verbeke@kuleuven.be}

\date{\today}


\begin{abstract}
We study almost inner derivations of $2$-step nilpotent Lie algebras
of genus $2$, i.e., having a $2$-dimensional commutator ideal, using matrix pencils.
In particular we determine all almost inner derivations of such algebras in terms of minimal indices and elementary divisors
over an arbitrary algebraically closed field of characteristic not $2$ and over the real numbers. 
 \end{abstract}

\maketitle

\section{Introduction}

One of the fundamental problems in spectral geometry is, to what extent the eigenvalues determine 
the geometry of a given manifold. A classical question here, going back to Hermann Weyl, asks whether 
or not isospectral manifolds need to be isometric. A first example with a negative answer was given in 
$1964$ by John Milnor. He constructed a pair of isospectral but non-isometric flat tori
of dimension $16$ using arithmetic lattices. Several other examples were given later and in $1984$ 
Gordon and Wilson \cite{GOW} even constructed continuous families of isospectral non-isometric manifolds. 
These manifolds were compact
Riemannian manifolds of the form $G/\Ga$ with a simply connected exponential solvable group $G$ and a discrete cocompact
subgroup $\Gamma$. The isospectral non-isometric property was derived from the existence of almost inner but non-inner
automorphisms of the Lie group $G$. Such automorphisms of $G$ can be studied on the Lie algebra level. They correspond
to {\em almost inner but non-inner derivations} of the Lie algebra $\Lg$ of $G$. \\[0.2cm]
Denote by  $\AID(\Lg)$ the subalgebra  of the derivation algebra $\Der(\Lg)$ consisting of almost inner derivations,
for a given finite-dimensional Lie algebra over a field $\K$. The motivation from spectral geometry and other applications
leads us to study the algebras $\AID(\Lg)$ in general.
Several classes of Lie algebras do not possess an almost inner but non-inner derivation.
Clearly any semisimple Lie algebra $\Lg$ over a field of characteristic zero satisfies $\AID(\Lg)=\Inn(\Lg)$. The same is true
for free-nilpotent Lie algebras or almost abelian Lie algebras, see \cite{BU61}, and for many other classes of Lie algebras.
Also all complex Lie algebras of dimension $n\le 4$ have no almost inner but non-inner derivations. In dimension $5$ however
there exist complex Lie algebras with almost inner but non-inner derivations. An example is the filiform nilpotent
Lie algebra with basis $\{x_1,\ldots ,x_5\}$ and brackets $[x_1,x_i]=x_{i+1}$ for $2\le i\le 4$ and $[x_2,x_3]=x_5$.
As it turns out, we can construct many almost inner but non-inner derivations for certain nilpotent Lie algebras and
in particular for certain {\em $2$-step nilpotent} Lie algebras. \\[0.2cm]
In this article we compute $\AID(\Lg)$ for all $2$-step nilpotent Lie algebras with a $2$-dimensional
commutator ideal over $\R$ and over an algebraically closed field $\K$ of characteristic not $2$.
Such Lie algebras can be described by skew-symmetric matrix pencils and their invariant polynomials
and elementary divisors due to Weierstrass and Kronecker. Using these canonical invariants helps a lot in computing
the almost inner derivations of such Lie algebras and we obtain an explicit formula. \\[0.2cm]
In section $3$ we review all necessary notions concerning matrix pencils, minimal indices and elementary divisors.
We give several examples of $2$-step nilpotent Lie algebras of genus $2$ with their associated invariants.
In section $4$ we introduce the subalgebra of central derivations $\CC(\Lg)$ in $\Der(\Lg)$, which contains
$\AID(\Lg)$. We give necessary and sufficient criteria for a derivation $D\in \CC(\Lg)$ to be almost inner.
Finally, in section $5$, we compute $\AID(\Lg)$ separately for each type of elementary divisor 
or minimal index and combine this for a general formula for $\AID(\Lg)$ in Theorems $\ref{5.10}$
and $\ref{5.11}$.

\section{Preliminaries}

Let $\Lg$ denote a Lie algebra over an arbitrary field $\K$ if not said otherwise. We will always assume that
$\Lg$ is finite-dimensional over $\K$.
The {\it lower central series} of $\Lg$ is given by
$\Lg^0=\Lg \supseteq \Lg^1 \supseteq \Lg^2 \supseteq \Lg^3  \supseteq \cdots$ where the ideals $\Lg^i$ are
recursively defined by $\Lg^i=[\Lg,\Lg^{i-1}]$ for all $i\ge 1$. A Lie algebra $\Lg$ is called {\em $c$-step nilpotent}
if $\Lg^c=0$. The {\em genus} of a Lie algebra $\Lg$ is the number $\dim(\Lg)-\abs{S}$, where $S$ is a minimal system
of generators. If $\Lg$ is nilpotent, then the genus is given by
$\dim (\Lg)-\dim \left(\Lg/[\Lg,\Lg]\right)=\dim([\Lg,\Lg])$. \\[0.2cm]
We recall the definition of an almost inner inner derivation \cite{GOW,BU55}.

\begin{defi}\label{aid}
A derivation $D\in \Der(\Lg)$ of a Lie algebra $\Lg$ is said to be {\em almost inner}, if $D(x)\in [\Lg,x]$ for all
$x\in \Lg$. The space of all almost inner derivations of $\Lg$ is denoted by $\AID(\Lg)$. 
\end{defi}

Let us denote by $\Inn(\Lg)$ the ideal of inner derivations in $\Der(\Lg)$. Certainly every inner derivation
of $\Lg$ is almost inner. The converse need not be true. So we obtain a chain of subalgebras
$\Inn(\Lg)\subset \AID(\Lg)\subset \Der(\Lg)$. \\[0.2cm]
Let $\Lg$ be a $2$-step nilpotent Lie algebra. Then $\Lg^2=[\Lg,[\Lg,\Lg]]=0$ so that $[\Lg,\Lg]\subseteq Z(\Lg)$.
If $\Lg$ has genus $1$, then $\Lg$ is a central extension of a $2n+1$-dimensional Heisenberg Lie algebra
and we have $\AID(\Lg) =\Inn(\Lg)$ by Lemma $3.3$ and Corollary $3.6$ of \cite{BU55}.
The next interesting case is genus $2$, which we will study here.

\section{Pencils, elementary divisors and minimal indices}

Two-step nilpotent Lie algebras have not been classified in general so far. For certain subclasses however
there is a complete description. We are interested in two-step nilpotent Lie algebras $\Lg$ of genus $2$, i.e.,
satisfying $\dim ([\Lg,\Lg])=2$. Such algebras can be described in terms of matrix pencils. This has been studied for
several purposes in the literature, see \cite{GAU,LET,GAN,THO,LAR,GAT,LAO}. We mainly follow the notation of \cite{GAU}.
Let $\K[\la,\mu]$ be the polynomial ring in two variables.

\begin{defi}
Let $A,B\in M_n(\K)$. A polynomial matrix $\mu A+\la B \in M_n(\K[\la,\mu])$ is called a {\em matrix pencil} or just a
{\em pencil}. Two such pencils $\mu A+\la B$ and $\mu C+\la D$ are called {\em strictly equivalent} if there are matrices
$S,T\in GL_n(\K)$ satisfying
\[
S(\mu A+\la B)T = \mu C+\la D.  
\]
The pencil is called {\em skew} if both $A$ and $B$ are skew-symmetric. Two skew-symmetric pencils
$\mu A+\la B$ and $\mu C+\la D$ are called {\em strictly congruent} if there is a matrix $S\in GL_n(\K)$
such that $S^t(\mu A +\la B)S=\mu C+\la D$. A pencil is called {\em regular} or {\em non-singular} if its determinant is not
the zero polynomial in $\K[\la,\mu]$.
\end{defi}  

It is known that skew-symmetric pencils over an algebraically closed field of characteristic different from two 
are strictly equivalent if and only if they are strictly congruent \cite{GAU}. The same is true over the field of
real numbers \cite{LAR}. \\[0.2cm]
Let $\Lg$ be a two-step nilpotent Lie algebra of genus $2$ with a basis $(x_1,\ldots ,x_n,y_1,y_2)$, where
$(y_1,y_2)$ is a basis of $[\Lg,\Lg]$. Denote by $A=(a_{ij})$ and $B=(b_{ij})$ the skew-symmetric matrices
of structure constants determined by
\[
[x_i,x_j]=a_{ij}y_1+b_{ij}y_2
\]
for all $1\le i,j\le n$. 

\begin{defi}
Let $\Lg$ be a two-step nilpotent Lie algebra of genus $2$ over an arbitrary field $\K$. Then
$\mu A+\la B \in M_n(\K[\la,\mu])$ is called the {\em pencil  associated to $\Lg$} with respect to a given basis
as above.
\end{defi}  

The following proposition is a special case of \cite[Proposition 4.1]{LET}

\begin{prop}\label{prop-iso}
Let $\Lg$ and $\Lh$ be two-step nilpotent Lie algebras of genus $2$ over an arbitrary field $\K$. If the pencils associated to
$\Lg$ and $\Lh$ with respect to some bases of $\Lg$ and $\Lh$ are strictly congruent, then $\Lg$ and $\Lh$ are isomorphic.
\end{prop}

From this proposition it follows that it is important for our study to be able to classify skew pencils up to strict equivalence.
For regular pencils this was solved by Weierstrass in terms of {\em elementary divisors}.
For a pencil $\mu A+\la B$ of rank $r$ let $G_m(\mu,\la)$ be the greatest
common divisor of all its minor determinants of order $m$. Then $G_{m}(\mu,\la)\mid G_{m+1}(\mu,\la)$ for all
$1 \le m\le r-1$. Let $i_1(\mu,\la)=G_1(\mu,\la)$ and
\[
i_m(\mu,\la)=\frac{G_{m}(\mu,\la)}{G_{m-1}(\mu,\la)}
\]  
for $1<m\le r$.

\begin{defi}
The homogeneous polynomials $\{i_m(\mu,\la)\}_m$ are called the {\em invariant factors} of the pencil $\mu A+\la B$.
Each polynomial $i_m(\mu,\la)$ can be written as a product of powers of prime polynomials because $\K[\la,\mu]$
is a unique factorization domain. These prime power factors (which are only determined up to scalar multiple) are called
the {\em elementary divisors} $e_a(\mu,\la)$, for $a=1,2,\ldots ,t$ of the pencil $\mu A+\la B$. An elementary divisor is
said to have {\em multiplicity $\nu$} if it appears exactly $\nu$ times in the factorizations of the invariant factors
$i_m(\mu,\la)$ for $1\le m\le r$.
\end{defi}  

Suppose that $\K$ is algebraically closed. In this case, the elementary divisors are all linear. Since the elementary
divisors are only determined up to scalar multiple, each elementary divisor is either of type $(b \mu+\la)^e$ or of type  $\mu^f$. 
The first one is called of {\em finite type}. The second one is called of {\em infinite type}, which means that the divisor
belongs to $\K[\mu]$. Elementary divisors of infinite type exist if and only if $\det(B)=0$. 
The elementary divisors $e_a(\mu,\la)$ of finite type correspond to the elementary divisors of the pencil  $A+\la B\in \K[\la]$
as follows. Setting $\mu=1$ in  $e_a(\mu,\la)$ we clearly obtain the elementary divisors $e_a(\la)$ of  $A+\la B$.
These can be computed by the {\em Smith normal form} because $\K[\la]$ is a PID. The diagonal elements of the Smith normal form
are just the invariant polynomials. Conversely, from each elementary divisor $e_a(\la)$ of $A+\la B$ of degree $e$
we obtain the corresponding elementary divisor $e_a(\mu,\la)$ by $e_a(\mu,\la)=\mu^ee_a(\frac{\la}{\mu})$.
In case $\K=\R$, apart from these elementary divisors of degree 1, there are also elementary divisors of degree 2 which are of
the form $(\lambda - \mu(a+b i))(\lambda - \mu (a-bi)) = \lambda^2 - 2 a \lambda \mu + (a^2+b^2) \mu^2$, with
$a\in \R,\; b \in \R^\times$.

\begin{ex}\label{3.4}
Let $\Lg$ be the $6$-dimensional real Lie algebra with basis $\{x_1,\ldots ,x_4,y_1,y_2\}$ and Lie brackets defined by
\begin{align*}
[x_1,x_2] & = y_1, \; [x_2,x_4]=y_2,\\
[x_1,x_3] & = y_2, \; [x_3,x_4]=-y_1. 
\end{align*}
Then the associated pencil is given by
\[
\mu A+\la B=\begin{pmatrix} 0 & \mu & \la & 0 \\ -\mu & 0 & 0 & \la \\
-\la & 0 & 0 & -\mu \\ 0 & -\la & \mu & 0 \end{pmatrix}.    
\]
The pencil is regular because $\det(\mu A+\la B)=(\mu^2+\la^2)^2$ is not the zero polynomial. The Smith normal form
of $A+\la B$ is given by $\diag(1,1,\la^2+1,\la^2+1)$. Hence there is one elementary
divisor $e_1(\mu,\la)=\mu^2(1+\frac{\la^2}{\mu^2})=\mu^2+\la^2$ of finite type of multiplicity $2$ and no elementary
divisor of infinite type.\\
When we consider the complexification $\Lg \otimes \C$ of this Lie algebra, then the corresponding pencil has two elementary
divisors: $\lambda - i \mu$ and $\lambda + i\mu$ both of multiplicity 2. 
\end{ex}  

For singular pencils we still need another invariant. Let $\mu A+\la B$ be a singular pencil of size $n$.
Then $(A+\la B)x=0$ has a nonzero solution in $\K[\la]^n$. Let $x_1(\la)$ be such a nonzero solution
of minimal degree $\ep_1$. Of all solutions which are $\K[\lambda]$-independent of  $x_1(\la)$ let us select a
solution  $x_2(\la)$ of minimal degree $\ep_2$. It is obvious that $\ep_1\le \ep_2$. By continuing this process
we obtain a set $x_1(\la),\ldots ,x_k(\la)$ of solutions, which is a maximal set of elements in $\K[\la]^n$
satisfying $(A+\la B)x_i(\la)=0$ for $i=1,\ldots ,k$ and being $\K[\la]$-independent. We have $k\le n$.
Note that this set is not uniquely determined, but that different sets have the same minimal degrees
$\ep_1\le \ep_2 \le \ldots \le \ep_k$. Hence the following notion is well-defined.

\begin{defi}
Let  $\mu A+\la B$ be a singular pencil. The associated numbers $\ep_1,\ldots ,\ep_k$ are called
the {\em minimal indices} of the pencil $\mu A+\la B$.
\end{defi}

\begin{ex}\label{3.6}
Let $\Lg$ be the $7$-dimensional real Lie algebra with basis $\{x_1,\ldots ,x_5,y_1,y_2\}$ and Lie brackets defined by
\begin{align*}
[x_1,x_3] & = y_1, \; [x_2,x_4]=y_1,\\
[x_1,x_4] & = y_2, \; [x_2,x_5]=y_2. 
\end{align*}
Then the associated pencil is given by
\[
\mu A+\la B=\begin{pmatrix} 0 & 0 & \mu & \la & 0 \\ 0 & 0 & 0 & \mu & \la \\
 -\mu & 0 & 0 & 0 & 0 \\ -\la & -\mu & 0 & 0 & 0 \\ 0 & -\la & 0 & 0 & 0 \end{pmatrix}.    
\]
Since $\det(\mu A+\la B)=0$ the pencil is singular. The equation $(A+\la B)x=0$ has a non-zero solution
$x_1(\la)=(0,0,\la^2,-\la,1)^t$, and the set is maximal. Hence there is one minimal index $\ep_1=2$. The Smith
normal form of $A+\la B$ is given by $\diag(1,1,1,1,0)$.
\end{ex}

The following well-known result classifies skew pencils up to congruence, see Corollary $6.6$ in \cite{GAU}
and Theorem $3.4$ in \cite{LET}.

\begin{prop}\label{3.7}
Let $\K$ be an algebraically closed field of characteristic not $2$ or the field of real numbers. Two skew-symmetric
pencils of the same dimension are strictly congruent if and only if they have the same elementary divisors and the
same minimal indices.
\end{prop}

For a skew pencil $\mu A+\la B$ over an algebraically closed field $\K$ the elementary divisors occur in pairs and we can arrange them as
\begin{align*}
\mu^{e_1},\mu^{e_1}, & \cdots , \mu^{e_s},\mu^{e_s}, \\
(\la-\mu \al_1)^{f_1}, (\la-\mu \al_1)^{f_1}, & \cdots ,  (\la-\mu \al_t)^{f_t}, (\la-\mu \al_t)^{f_t} 
\end{align*}
where $\al_1, \, \al_2, \ldots, \al_t\in \K$. 

When $\K=\R$, the elementary divisors still occur in pairs, and apart from the above set of elementary 
divisors (where of course $\al_1,\ldots ,\al_t\in \R$), we can also have pairs of the form 
\[
\xi (a_1,b_1)^{m_1},\xi (a_1,b_1)^{m_1}, \cdots , \xi (a_p,b_p)^{m_p},\xi (a_p,b_p)^{m_p},  
\]
where $a_1,\ldots ,a_p\in \R$, $b_1,\ldots ,b_p\in \R^{\times}$ and $\xi(a,b)=(\la-\mu(a+ib))(\la-\mu(a-bi))$
with $a,b\in \R$.

Since elementary divisors occur in pairs, we introduce a notation to indicate such pairs. Let $\K$ be an algebraically closed
field or $\K=\R$.
For given $\al\in \K\cup \{ \infty\}$ or $\alpha \in \C\cup \{\infty\}$ in case $\K=\R$ and $e\in \N$ we denote by
$(\al,e)$ the following pairs of elementary divisors
\[
(\al,e):=\begin{cases}
\mu^e, \mu^e & \text{ if } \alpha = \infty \\
(\lambda - \mu \alpha)^e, (\lambda - \mu \alpha)^e & \text{ if } \alpha \in \K\\
\xi(a,b)^e,\xi(a,b)^e  & \text{ if $\K=\R$ and }\alpha = a + b i \in \C \setminus \R.
\end{cases}    
\]  

Hence for a skew pencil $\mu A+\la B$ over an algebraically closed field $\K$ we can associate in a unique way a set of
elementary divisors as follows:
\[
(\infty,e_1),  \ldots , (\infty, e_s), \; (\al_1,f_1),  \ldots , (\al_t,f_t)\]
with $\al_1, \ldots, \al_t \in \K$ and for a skew pencil $\mu A+\la B$ over $\R$ we find a set of elementary divisors of the form
\[
(\infty,e_1),  \ldots , (\infty, e_s), \; (\al_1,f_1),  \ldots , (\al_t,f_t), \; (\be_1,m_1), \ldots , (\be_p,m_p),
\]  
where $\al_i\in \R$ and $\be_i\in \C\setminus \R$.

For a giving pair of elementary divisors $(\al,e)$ as above or a minimal index $\ep$, there exists a canonical skew pencil
having exactly that one pair of elementary divisors $(\al,e)$ (and no minimal indices or other elementary divisors) or 
having no elementary divisors and exactly one minimal index $\ep$. \\[0.2cm]
These skew pencils are given by the following cases: \\[0.2cm]
{\em Case 1:} For $(\al,e)=(\infty,e)$ the skew pencil is given by
\[
F(\infty,e):= \begin{pmatrix}
        0 & \mu \Delta_e + \lambda \Lambda_e\\
        -\mu \Delta_e - \lambda \Lambda_e & 0
        \end{pmatrix} \in M_{2e}(\K[\mu, \lambda]),
\]
where
\[
\Delta_e = \begin{pmatrix}
        & & & & 1 \\
        & & & 1 & \\
        & & \iddots & & \\
        & 1 & & & \\
        1 & & & &
      \end{pmatrix},\quad
      \Lambda_e = \begin{pmatrix}
        & & & & 0 \\
        & & & 0 & 1 \\
        & & \iddots & 1 & \\
        & 0 & \iddots & & \\
        0 & 1 & & & 
        \end{pmatrix} \in M_e(\K).
\]
{\em Case 2:} For $(\al,f)=(\la-\mu \al)^f, (\la-\mu\al)^f$ the skew pencil is given by
\[
F(\alpha,f) := \begin{pmatrix}
       0 & (\lambda - \mu \alpha) \Delta_f + \mu \Lambda_f \\
        -(\lambda - \mu \alpha) \Delta_f - \mu \Lambda_f & 0
        \end{pmatrix}  \in M_{2f}(\K[\mu, \lambda]). 
\]
{\em Case 3:} Only in case $\K=\R$ and for $(\al,m)=(a+b i, m)=\xi(a,b)^m,\xi(a,b)^m$ (with $a+b i \in \C\setminus \R$)
the real skew pencil is given by
\[
C(a,b,m):=\begin{pmatrix} 0 & T_m \\ -T_m & 0 \end{pmatrix}\in M_{4m}(\R[\mu,\la]),
\]
where
\[
T_m=  
 \begin{pmatrix}
        & & & 0 & R \\
        & & \iddots & R & \mu \Delta_2 \\
        & & \iddots & & & \\
        0 & R & \iddots & & & \\
        R & \mu \Delta_2 & & & 
        \end{pmatrix} \in M_{2m}(\R[\mu, \lambda])
\]  
for $m\ge 2$ and
\[
T_1=R= \begin{pmatrix} - \mu b & \lambda - \mu a \\ \lambda - \mu a & \mu b \end{pmatrix} \in M_2(\R[\mu,\la]).
\]  
{\em Case 4:} For each minimal index $\ep\geq 1$ the skew pencil is given by
\[
M_\varepsilon = \begin{pmatrix}
        0_{\varepsilon + 1} & \mathcal{L}_\varepsilon(\mu, \lambda) \\
        -\mathcal{L}_\varepsilon(\mu, \lambda)^t & 0_\varepsilon
        \end{pmatrix} \in M_{2 \varepsilon + 1}(\K[\mu, \lambda]), 
\]  
where
\[
\mathcal{L}_\varepsilon(\mu, \lambda) = \begin{pmatrix}
        \lambda & 0 & \ldots & 0 \\
        \mu  & \lambda & & 0 \\
        & \ddots & \ddots & \\
        0 & & \mu  & \lambda \\
        0 & \ldots & 0 & \mu 
        \end{pmatrix} \in M_{\varepsilon + 1, \varepsilon}(\K[\mu, \lambda]).
\]  
For minimal index $\ep=0$, the skew pencil is just $M_0=(0)$, the $1\times 1$ zero matrix.

\medskip

Using the notations of above we have the following result, see \cite{GAU,LET}.

\begin{prop}\label{canon-form}
Let $\K$ be an algebraically closed field of characteristic not 2. Any skew pencil over $\K$ with elementary divisors 
$(\infty,e_1),  \ldots , (\infty, e_s), \; (\al_1,f_1),  \ldots , (\al_t,f_t)$
and minimal indices $\ep_1,\ldots ,\ep_k$ is strictly congruent to the pencil consisting of a matrix with the blocks
\[
F(\infty,e_1),  \ldots ,F(\infty,e_s),\; F(\al_1,f_1),\ldots ,F(\al_t,f_t),\; M_{\ep_1},\ldots ,M_{\ep_k}  
\]  
on the diagonal. Any skew pencil over $\R$  having elementary divisors
\[
(\infty,e_1),  \ldots ,\; (\infty, e_s), \; (\al_1,f_1),  \ldots , (\al_t,f_t),\;(a_1 +b_1 i,m_1), \ldots , (a_p+b_p i,m_p),
\]  
and minimal indices $\ep_1,\ldots ,\ep_k$ is strictly congruent to the pencil consisting of a matrix with the blocks
\[F(\infty,e_1), \ldots ,F(\infty,e_s),\; F(\al_1,f_1),\ldots ,F(\al_t,f_t),\]
\[C(a_1,b_1,m_1), \ldots ,C(a_p,b_p,m_p),\; M_{\ep_1},\ldots ,M_{\ep_k}  
\]
on the diagonal. 
\end{prop}

\begin{defi}\label{3.9}
A $2$-step nilpotent Lie algebra of genus $2$ over an algebraically closed field $\K$ of characteristic not 2 or
$\K=\R$ is called {\em canonical} if its associated skew pencil is of a blocked diagonal form as in
Proposition~\ref{canon-form} above.  
\end{defi}

As an immediate consequence of Proposition~\ref{prop-iso} we obtain the following corollary:
\begin{cor}
Let $\K$ be an algebraically closed field of characteristic not 2 or $\K=\R$.\\
Any $2$-step nilpotent Lie algebra of genus $2$ over $\K$ is isomorphic to a canonical one with the same
elementary divisors and minimal indices. 
\end{cor}

It follows that the computation of $\AID(\Lg)$ for $2$-step nilpotent Lie algebras of
genus $2$ over $\K$ can be reduced to canonical Lie algebras.

\section{Applications to almost inner derivations}

Any almost inner derivation $D$ of a $2$-step nilpotent Lie algebra $\Lg$ maps the center $Z(\Lg)$ to zero and
$\Lg$ to $[\Lg,\Lg]$, see Proposition $2.9$ in \cite{BU55}. Hence the space $\CC(\Lg)$ defined below contains
$\AID(\Lg)$ and is a subalgebra of $\Der(\Lg)$. 

\begin{defi}
Let $\Lg$ be a $2$-step nilpotent Lie algebra over a field $\K$. Then
\[
\CC(\Lg)=\{D\in \End(\Lg)\mid D(Z(\Lg))=0, D(\Lg)\subseteq [\Lg,\Lg]\}
\]
is a subalgebra of $\Der(\Lg)$ with $\AID(\Lg)\subseteq \CC(\Lg)$. It is called the algebra of {\em central derivations}.
\end{defi}

Since $[\Lg,\Lg]\subseteq Z(\Lg)$ for $2$-step nilpotent Lie algebras, any $D\in \CC(\Lg)$ is a derivation
of $\Lg$. \\[0.2cm]
Let us now assume that $\Lg$ be a $2$-step nilpotent Lie algebra of genus $2$. We fix a basis
\[
\{x_1,\ldots ,x_n,y_1,y_2\},
\]
where $y_1,y_2$ spans $[\Lg,\Lg]$. Let $\mu A+\la B$ be the associated pencil.
Every element in $v\in \Lg$ can be written as $v=x+y$ in this basis, where
$x=a_1x_1+\ldots ,a_nx_n$ and $y=b_1y_1+b_2y_2$ for $a_i,b_i\in \K$. For every $D\in \CC(\Lg)$ we have $D(y)=0$ and
$D(x)=d_1(x)y_1+d_2(x)y_2$ for some $d_1,d_2\in \Hom(U,\K)$, where $U=\langle x_1, x_2, \ldots, x_n\rangle$.
Recall that $D\in \AID(\Lg)$ if and only if there exists
a map $\phi_D\colon \Lg\ra \Lg$ such that
\[
D(v)=[v,\phi_D(v)]
\]  
for all $v\in \Lg$. We may assume that $\phi_D(v)=\phi_D(x)\in \langle x_1,\ldots ,x_n\rangle$ for $v=x+y$ as above, i.e.,
that $v=x$ and
\[
\phi_D(x)=c_1(x)x_1+\cdots +c_n(x)x_n
\]  
for some $c_i(x)\in \K$ for $i=1,\ldots ,n$. Let
\begin{align*}
c(x) & =(c_1(x),\ldots ,c_n(x))^t, \\
a(x) & = (a_1,\ldots ,a_n)^t,\\
L(x) & = \begin{pmatrix} a(x)^t  A \\ a(x)^t  B \end{pmatrix}\in M_{2,n}(\K),\\
d(x) & = \begin{pmatrix} d_1(x) \\ d_2(x)\end{pmatrix}.
\end{align*}

\begin{lem}\label{4.2}
Let $\Lg$ be a $2$-step nilpotent Lie algebra of genus $2$ over $\K$ with the notations as above. Then a given map
$D\in \CC(\Lg)$ is in $\AID(\Lg)$ if and only if $L(x)c(x)=d(x)$ has a solution in the unknowns $c_i(x)$, $i=1,\ldots ,n$
for all $x=a_1x_1+\cdots +a_nx_n$ in $\Lg$. 
\end{lem}

\begin{proof}
We have $D\in \AID(\Lg)$ if and only if for all $x=a_1x_1+\cdots +a_nx_n$ there exists
a map $\phi_D\colon \Lg\ra \Lg$ such that $D(x)=[x,\phi_D(x)]$, i.e., if and only if we can find $c_i(x)$ for $i=1,\ldots ,n$
such that
\[
D(x)=\left[ \sum_{i=1}^n a_ix_i, \sum_{j=1}^n c_j(x)x_j\right]=\sum_{i=1}^n \sum_{j=1}^na_ic_j(x)(a_{ij}y_1+b_{ij}y_2).
\]
This is equivalent to the system of linear equations $L(x)c(x)=d(x)$ for all $x=a_1x_1+\cdots +a_nx_n$.
\end{proof}  

We obtain the following result.

\begin{prop}\label{4.3}
Let $\Lg$ be a $2$-step nilpotent Lie algebra of genus $2$ over $\K$ with associated pencil $\mu A+\la B$ and
the notations as above.  Then a given map $D\in \CC(\Lg)$ is in $\AID(\Lg)$ if and only if for all
$x=a_1x_1+\cdots +a_nx_n$ in $\Lg$ and all $\la,\mu \in \K$ we have the condition
\begin{align}\label{c1}
(\mu A+\la B)a(x)=0 & \Longrightarrow \mu d_1(x)+\la d_2(x)=0.
\end{align}  
\end{prop}  

\begin{proof}
Suppose that $D\in \AID(\Lg)$ and  $(\mu A+\la B)a(x)=0$. Since  $A$ and $B$ are skew-symmetric, we have
\[  
 0=(\mu A+\la B)a(x)=-(\mu a(x)^tA+\la a(x)^tB)^t,
\]
so that $\mu a(x)^tA+\la a(x)^tB=0$. Now $a(x)^tA$ is the first row of $L(x)$ and $a(x)^tB$ the second row of it.
Since $D\in \AID(\Lg)$, both $L(x)$ and the extended matrix $(L(x)\mid d(x))$ have the same rank by Lemma $\ref{4.2}$.
Hence for any linear combination of rows of $L(x)$ which equals zero, the same linear combination of rows of $(L(x)\mid d(x))$
equals zero. Hence $\mu d_1(x)+\la d_2(x)=0$. The converse direction follows similarly.
\end{proof}  

We will apply this result to Example $\ref{3.6}$:

\begin{ex}\label{4.4}
Let $\Lg$ be the $7$-dimensional real Lie algebra with basis $\{x_1,\ldots ,x_5,y_1,y_2\}$ and Lie brackets defined by
\begin{align*}
[x_1,x_3] & = y_1, \; [x_2,x_4]=y_1,\\
[x_1,x_4] & = y_2, \; [x_2,x_5]=y_2. 
\end{align*}
Let $D\in \CC(\Lg)$ and $x=a_1x_1+\cdots +a_5x_5$. Then the matrix of $D$ is given by
\[
D=
\begin{pmatrix}
 \bigzero  & \rvline & \bigzero \\
\hline
\begin{matrix} r_1 & r_2 & r_3 & r_4 & r_5 \\
 s_1 & s_2 & s_3 & s_4 & s_5 \end{matrix}
  & \rvline & \bigzero
 \end{pmatrix}
\]
and $d_1(x)=a_1r_1+\cdots +a_5r_5$, $d_2(x)=a_1s_1+\cdots +a_5s_5$. Condition \eqref{c1} then yields that $D\in \AID(\Lg)$ if
and only if $r_5=s_3=0$ and $s_5=r_4$, $s_4=r_3$. Thus we have $\dim (\AID(\Lg))=6$ and $\dim \Inn(\Lg)=5$.
\end{ex}  
In fact, the kernel of $\mu A+\la B$ for this Lie algebra is $1$-dimensional for all $(\mu,\la)\neq (0,0)$, generated by
$a(x)=(0,0,\la^2,-\mu\la,\mu^2)^t$. Therefore condition \eqref{c1} applied to this vector yields 
\begin{align*}
  0 & =\mu d_1(x)+\la d_2(x) \\
    & = \mu(a_1r_1+\cdots +a_5r_5)+\la(a_1s_1+\cdots +a_5s_5) \\
    & = \mu^3r_5 + \mu^2\la( s_5 - r_4) + \mu\la^2(r_3 - s_4) + \la^3s_3.
\end{align*}  
for all $\la,\mu\in \K$, which shows the claim.
 
\begin{cor}\label{4.5}
Let $\Lg$ be a $2$-step nilpotent Lie algebra of genus $2$ over $\R$, whose associated pencil $\mu A+\la B$ satisfies
$\det(\mu A+\la B)=0$ if and only if $\mu=\la=0$. Then $\AID(\Lg)=\CC(\Lg)$ and 
$\dim (\AID(\Lg))=2\dim (\Inn (\Lg))=2n$.  
\end{cor}

\begin{proof}
By assumption the system $(\mu A+\la B)a(x)=0$ only has the trivial solution $x=0$. Hence
condition \eqref{c1} is satisfied and it follows $\AID(\Lg)=\CC(\Lg)$ from Proposition $\ref{4.3}$.
Since the assumptions imply that $[\Lg,\Lg]=Z(\Lg)$ we have
$\dim \Inn (\Lg)=n$, spanned by $\ad(x_1),\ldots ,\ad(x_n)$, and $\dim \CC(\Lg)=2n$. It follows that
 $\dim (\AID(\Lg))=2n$.
\end{proof}  

We can apply this corollary to Example $\ref{3.4}$ with $n=4$.

\begin{ex}\label{4.6}
Let $\Lg$ be the Lie algebra over a field $\K$ with basis $\{x_1,\ldots ,x_4,y_1,y_2\}$ and Lie brackets defined by
\begin{align*}
[x_1,x_2] & = y_1, \; [x_2,x_4]=y_2,\\
[x_1,x_3] & = y_2, \; [x_3,x_4]=-y_1. 
\end{align*}
Then for $\K=\R$, we have that $\det(\mu A +\la B) = (\mu^2+\la^2)^2 =0$ if and only if $\mu=\la=0$, so
that $\AID(\Lg)=\CC(\Lg)$ and $\dim (\AID(\Lg))=2\cdot \dim (\Inn(\Lg))=8$.
However, we have  $\AID(\Lg)=\Inn(\Lg)$ for $\K$ being an algebraically closed field of characteristic not 2.
This will follow from our main result Theorem $\ref{5.11}$.
\end{ex}

\section{Almost inner derivations of Lie algebras of genus 2}

In this section we determine the algebra $\AID(\Lg)$ for canonical Lie algebras in the sense of Definition
$\ref{3.9}$ over $\K$, where $\K$ is either $\R$ or an algebraically closed field of characteristic not 2.
We will start with the case that the canonical pencil only consists of one block.

\begin{defi}
Let $\Lg$ be a Lie algebra with a given basis $\CB$. We say that a linear map $D\colon \Lg\ra \Lg$ is
{\em $\CB$-almost inner}, if $D(x)\in [x,\Lg]$ for all basis elements $x\in \CB$.  
\end{defi}  

A derivation, which is $\CB$-almost inner for some basis $\CB$ need not be almost inner.

\begin{lem}\label{5.2}
Let $\Lg$ be a canonical Lie algebra over $\K$ with one pair of elementary divisors $(\infty,e)$. Then we have
$\dim (\Inn(\Lg))=2e$ and $\dim (\AID(\Lg))=4e-2$.
\end{lem}  

\begin{proof}
By assumption the matrix pencil of $\Lg$ is $F(\infty,e)$, so that the Lie brackets of $\Lg$ in the usual basis
$\CB$ are given by
\begin{align*}
[x_i,x_{2e+1-i}]& = y_1,\; 1\le i\le e,\\
[x_j,x_{2e+2-j}]& = y_2,\; 2\le j\le e.  
\end{align*}
We have $\dim (\Inn(\Lg))=2e$. We will compute $\AID(\Lg)$ by Proposition $\ref{4.3}$.
A basis for $\CC(\Lg)$ is given by the maps $D_{i,j}:\Lg\ra \Lg$ for $1\le i\le 2e$ and $j=1,2$ defined by
\[
\sum_{k=1}^{2e} a_kx_k+(b_1y_1+b_2y_2) \mapsto a_iy_j.
\]
We have $\dim \CC(\Lg)=4e$. 
It is easy to see that all $D_{i,j}$ are $\CB$-almost inner except for $D_{1,2}$ and $D_{e+1,2}$.
It is easy to see that the span of $D_{1,2}$ and $D_{e+1,2}$ has only trivial intersection with $\Inn(\Lg)$.
Then we have $\dim \AID(\Lg)=4e-2$ if we can show that all of the remaining $4e-2$ derivations are actually almost inner. \\
Let $D\in \CC(\Lg)$ be $\CB$-almost inner. We have $\det(\mu A+\la B)=\mu^{2e}$. For $\mu\neq 0$ condition \eqref{c1} is
satisfied, so that we may assume that $\mu=0$. Then the kernel of $\mu A+\la B=\la B=F(\infty,e)$ is equal to the
set of all vectors $a(x)=(k_1,0,\ldots ,0,k_{e+1},0,\ldots ,0)^t$ with $k_1,k_{e+1}\in \K$. For these vectors we have $d_2(a(x))=0$,
so that condition  \eqref{c1} is satisfied and the proof is finished.
\end{proof}  

\begin{lem}\label{5.3}
Let $\Lg$ be a canonical Lie algebra over $\K$ with one pair of elementary divisors $(\al,f)$. Then we have
$\dim (\Inn(\Lg))=2f$ and $\dim (\AID(\Lg))=4f-2$.
\end{lem}  

\begin{proof}
The Lie brackets of $\Lg$ with respect to the usual basis $\{x_1,\ldots ,x_{2f},y_1,y_2\}$ and matrix
pencil $F(\al,f)$ are given by
\begin{align*}
[x_i,x_{2f+1-i}]& =y_2-\al y_1,\; 1\le i\le f,\\
[x_j,x_{2f+2-j}]& =y_1,\; 2\le j\le e.  
\end{align*}
We may pass to the basis $\{x_1,\ldots ,x_{2f},y_2-\al y_1,y_1\}$ so that $\Lg$ coincides with the Lie algebra
of Lemma $\ref{5.2}$. This finishes the proof.
\end{proof}

The next Lemma is only for the case $\K=\R$.

\begin{lem}\label{5.4}
Let $\Lg$ be a canonical Lie algebra over $\R$ with one pair of elementary divisors $(\be,m)$, where $\be=a+bi$
and $b\neq 0$.  Then we have $\dim \Inn(\Lg)=4m$ and $\dim \AID(\Lg)=8m$.
\end{lem}  

\begin{proof}
The Lie brackets of $\Lg$ with respect to the usual basis $\{x_1,\ldots ,x_{4m},y_1,y_2'\}$ and matrix
pencil $\mu A +\la B=C(a,b,m)$ are given by
\begin{align*}
[x_{2i-1},x_{4m-2i+1}]& =-b y_1,\; 1\le i\le m,\\
[x_{2i},x_{4m-2i+2}] & =by_1, \; 1\le i\le m,\\
[x_j,x_{4m+1-j}]& =y_2'-ay_1,\; 1\le j\le 2m,  
\end{align*}
and in addition, for $m\ge 2$, 
\[
[x_k,x_{4m-k+3}]=y_1,\; 3\le k\le 2m.
\]  
We can pass to a basis $\CB=\{x_1,\ldots ,x_n,y_1,y_2 \}$ by setting $y_2:=y_2'-ay_1$.
Then a basis for $\CC(\Lg)$ is given by the maps
$D_{i,j}:\Lg\ra \Lg$ for $1\le i\le 4m$ and $j=1,2$ defined by
\[
\sum_{k=1}^{4m} a_kx_k+(b_1y_1+b_2y_2) \mapsto a_iy_j.
\]
We have $\dim (\CC(\Lg))=8m$ and $\det(\mu A+\la B)=(\la^2+\mu^2b^2)^{2m}$, so that $\AID(\Lg)=\CC(\Lg)$
over $\R$ by Corollary $\ref{4.5}$.
\end{proof}

Let us again consider Example $\ref{3.4}$.

\begin{ex}\label{5.5}
Let $\Lg$ be the $6$-dimensional Lie algebra over $\R$ with basis $\{x_1,\ldots ,x_4,y_1,y_2\}$ and Lie brackets defined by
\begin{align*}
[x_1,x_2] & = y_1, \; [x_2,x_4]=y_2,\\
[x_1,x_3] & = y_2, \; [x_3,x_4]=-y_1. 
\end{align*}
This is a canonical Lie algebra with one pair of elementary divisors $(\be,m)=(i,1)$. Hence, Lemma \ref{5.4} says that
$\dim (\Inn(\Lg))=4$ and $\dim (\AID(\Lg))=8$ which coincides with what we obtained in Example~\ref{4.6}.
\end{ex}

\begin{lem}\label{5.6}
Let $\Lg$ be a canonical Lie algebra over $\K$ with minimal index $\ep \ge 1$. Then it holds
$\dim (\Inn(\Lg))=2\ep +1$ and $\dim (\AID(\Lg))=3\ep$.
\end{lem}  

\begin{proof}
The Lie brackets of $\Lg$ with respect to the usual basis $\CB=\{ x_1,\ldots ,x_{2\ep+1},y_1,y_2\}$ and matrix
pencil $\mu A+\la B=M_{\ep}$ are given by
\begin{align*}
[x_i,x_{i+\ep+1}]& =y_2,\; 1\le i\le \ep,\\
[x_{j+1},x_{j+\ep+1}]& =y_1,\; 1\le j\le \ep.  
\end{align*}
It is easy to see that $Z(\Lg)=\langle y_1, y_2 \rangle$, so we have 
$\dim (\Inn(\Lg))=2\ep+1$. For $\ep=1$ we have $\AID(\Lg)=\Inn(\Lg)$ by Theorem $4.1$ of \cite{BU55}
since $\Lg$ is determined by a graph. For $\ep\ge 2$ a basis of $\CC(\Lg)$ is given by the maps
$D_{i,j}:\Lg\ra \Lg$ for $1\le i\le 2\ep+1$ and $j=1,2$ defined by
\[
\sum_{k=1}^{2\ep+1} a_kx_k+(b_1y_1+b_2y_2) \mapsto a_iy_j.
\]
Hence we have $\dim (\CC(\Lg))=4\ep+2$ and $\AID(\Lg)\subseteq \CC(\Lg)$. Suppose now that
\[
D=\sum_{i=1}^{2\ep+1}\al_iD_{i,1}+\sum_{i=1}^{2\ep+1}\be_iD_{i,2}
\]  
is an element of $\AID(\Lg)$. Then for any $b\in \K$ we have
\begin{align}\label{c2}
D\left(\sum_{i=1}^{\ep+1}b^ix_i\right) & = \sum_{i=1}^{\ep+1}\al_ib^iy_1+\sum_{i=1}^{\ep+1}\be_ib^iy_2.
\end{align}
Since $D\in \AID(\Lg)$ there exist $c_j(b)\in \K$ for all $\ep+2\le j\le 2\ep+1$ such that
\begin{align}\label{c3}
D\left(\sum_{i=1}^{\ep+1}b^ix_i\right) & =\left[\sum_{i=1}^{\ep+1}b^ix_i,\sum_{j=\ep+2}^{2\ep+1}c_j(b)x_j\right]
= \sum_{i=2}^{\ep+1}b^ic_{i+\ep}(b)y_1+\sum_{i=1}^{\ep}b^ic_{i+\ep+1}(b)y_2.                                          
\end{align}  
We also have
\begin{align}\label{c4}
b\left(\sum_{i=1}^{\ep}b^i c_{i+\ep+1}(b)\right) & = \sum_{i=2}^{\ep+1}b^ic_{i+\ep}(b).
\end{align}  
Comparing coefficients of $y_1$ and $y_2$ in \eqref{c2} and \eqref{c3} and using \eqref{c4} we find that
\[
b\sum_{i=1}^{\ep+1}\be_ib^i -\sum_{i=1}^{\ep+1} \al_ib^i=0,
\]
so that
\[
-\al_1b+\sum_{i=2}^{\ep+1}(\be_{i-1}-\al_i)b^i+\be_{\ep+1}b^{\ep+2}=0.
\]  
Since this holds for all $b\in \K$ we obtain
\[
\al_1=0,\; \be_{\ep+1}=0,\; \al_i=\be_{i-1} \text{ for } 2\le i\le \ep+1.
\]  
This means that $\AID(\Lg)$ is contained in the subspace
\[
V=\{D=\sum_{i=1}^{2\ep+1}\al_iD_{i,1}+\sum_{i=1}^{2\ep+1}\be_iD_{i,2}\in \CC(\Lg)\mid \al_1=\be_{\ep+1}=0,\al_i=\be_{i-1}
 \text{ for } 2\le i\le \ep+1 \}.
\]  
Note that $\dim (V)=\dim (\CC(\Lg))-(\ep+2)=4\ep+2-(\ep+2)=3\ep$. Hence $\dim (\AID(\Lg))\le 3\ep$.
We claim that there holds equality. More precisely we will show that each $D_{j,1}$ for $\ep+2\le j\le 2\ep$ is almost
inner. Here we do not consider $D_{2\ep+1,1}$ because it already coincides with the inner derivation $\ad (x_{\ep+1})$
and hence is almost inner.
Let
\[
x = \sum_{i = 1}^{2 \varepsilon + 1} a_i x_i + (b_1 y_1 + b_2 y_2)
\]
be an element in $\Lg$. If $a_j = 0$, then $D_{j,1}(x) = [x,0] = 0$. Otherwise we have 
\[ 
D_{j ,1}(x) = \left[x, \frac{-a_j}{a_\ell} x_{\ell - \varepsilon}\right] = a_j y_1
\]
for $\ell := \max\{j \leq k \leq 2 \varepsilon +1 \mid a_k \neq 0\}$. This shows that $D_{j,1}$ is almost inner for
all $\varepsilon + 2 \leq j \leq 2\varepsilon$. Consider the subspace $W$ of $\AID(\Lg)$ generated by all 
$D_{j,1}$ for $\varepsilon + 2 \leq j \leq 2\varepsilon$. We claim that
\[
W\cap \Inn(\Lg)=0.
\]  
Then we are done. We know that $\dim(\Inn(\Lg))=2\ep+1$ and $\dim (W)=\ep-1$, so that $\Inn(\Lg)\oplus W$ is a
$3\ep$-dimensional subspace of $\AID(\Lg)$. This implies $\dim (\AID(\Lg))\ge 3\ep$ and hence there holds equality.
So assume that $D=\sum_{i=\ep+2}^{2\ep}\al_iD_{i,1}\in W\cap \Inn(\Lg)$ with $D=\ad(x)$ for some $x=\sum_{i=1}^{2\ep+1}k_ix_i$.
We will show that all $k_i=0$ and all $\al_i=0$. Because of
\[
0=\ad(x)(x_j)=D(x_j)=\left[\sum_{i=1}^{2\ep+1}k_ix_i,x_j\right]=-k_{\ep+j+1}y_2-k_{\ep+j}y_1
\]
for $2\le j\le \ep$ we have $k_{\ep+2}=k_{\ep+3}=\cdots =k_{2\ep+1}=0$. 
Also we have
\[
\al_{\ep+j}y_1=D(x_{\ep+j})=[k_1x_1+\cdots +k_{\ep+1}x_{\ep+1},x_{\ep+j}]=k_{j-1}y_2+k_{j}y_1
\]  
for all $2\le j\le \ep+1$, where we take $\al_{2 \ep +1}=0$. Hence all $\al_i$ and all $k_i$ are zero and we have shown
that $W\cap \Inn(\Lg)=0$.
\end{proof}  

\begin{rem} For a minimal index $\ep=0$, the corresponding Lie algebra $\Lg$ is just the abelian 3-dimensional
Lie algebra (with basis $x_1,y_1,y_2$) over $\K$ and so in this case $\Inn(\Lg)=\AID(\Lg)=0$. 
\end{rem}

\begin{ex}\label{5.7}
For $\ep=2$ the canonical Lie algebra $\Lg$ of Lemma $\ref{5.6}$ is isomorphic to the Lie algebra of Example $\ref{4.4}$.
We have $\dim (\Inn(\Lg))=2\ep+1=5$ and $\dim (\AID(\Lg))=3\ep=6$, which coincides with the result of Example $\ref{4.4}$.
\end{ex}  

For the next lemma, let $\Lg$ be a $2$-step nilpotent Lie algebra over an arbitrary field $\K$ with basis
$\{ x_1,\ldots ,x_n,y_1,\ldots ,y_m,z_1,\ldots ,z_p\}$, where the $z_i$ span $[\Lg,\Lg]$. Define Lie subalgebras by
\begin{align*}
\Lg_x & = \langle  x_1,\ldots ,x_n,z_1,\ldots ,z_p \rangle,\\
\Lg_y & = \langle  y_1,\ldots ,y_m,z_1,\ldots ,z_p \rangle.  
\end{align*}  

\begin{lem}\label{5.8}
Let $\Lg$ be a $2$-step nilpotent Lie algebra over a field $\K$ with the above basis such that $[x_i,y_j]=0$
for all $1\le i\le n$ and $1\le j\le m$. Then we have
\[
\dim (\AID(\Lg))=\dim (\AID(\Lg_x))+\dim (\AID(\Lg_y)).
\]  
\end{lem}  

\begin{proof}
 Let $D\in \AID(\Lg)$ and write $e=x+y+z$, where $x$ is a linear combination of the $x_i$, $y$ is a linear
 combination of the $y_i$ and $z$ a linear combination of the $z_i$. Then there are maps
$\varphi_{D_x}: \Lg_x \to \Lg_x$ and $\varphi_{D_y}: \Lg_y \to \Lg_y$ such that
\[
D(e) = D(x + y + z) = \left[x, \varphi_{D_x}(x)\right] + \left[y, \varphi_{D_y}(y)\right],
\] 
This means that $D_x: \Lg_x \to \Lg_x,\; x \mapsto D_{\mid_{\Lg_x}}(x) \in \AID(\Lg_x)$ with determination map
$\varphi_{D_x}$ and $D_y: \Lg_y \to \Lg_y,\; y \mapsto D_{\mid_{\Lg_y}}(y) \in \AID(\Lg_y)$ determined by $\varphi_{D_y}$.
Conversely, any almost inner derivation of $\Lg_x$ or $\Lg_y$ can be extended to an almost inner derivation of $\Lg$.
\end{proof}  

Finally we can state our main result of this paper by combining the previous lemmas. For clarity, we formulate this result
as two separate theorems depending on the type of field $\K$ we are considering. We only give a proof for the last theorem
in case $\K=\R$. The proof for the other case is similar.

\begin{thm}\label{5.10}
Let $\Lg$ be a $2$-step nilpotent Lie algebra of genus $2$ over an algebraically closed field
$\K$ of characteristic not 2 with minimal indices $\varepsilon_1,\ldots,\varepsilon_k$ and elementary divisors
\[
  (\infty,e_1),\ldots,(\infty,e_s),\; (\alpha_1,f_1), \ldots, (\alpha_l,f_l)
\]  
with $\alpha_j \in \K$ for all $1 \leq j \leq l$.
Then we have
\[        
  \dim(\AID(\Lg)) = \dim(\Inn(\Lg)) + \sum_{j = 1, \ep_j\neq 0}^k (\varepsilon_j - 1) + 2\sum_{j = 1}^s (e_j - 1)
  + 2\sum_{j = 1}^l (f_l - 1). 
\]
\end{thm}

\begin{thm}\label{5.11}
Let $\Lg$ be a $2$-step nilpotent Lie algebra of genus $2$ over $\R$ with minimal indices
$\varepsilon_1,\ldots,\varepsilon_k$ and elementary divisors
\[
  (\infty,e_1),\ldots,(\infty,e_s),\; (\alpha_1,f_1), \ldots, (\alpha_l,f_l),\;
  (\beta_1,m_1),\ldots,(\beta_p,m_p)
\]  
with $\alpha_j \in \R$ and $\beta_r = a_r + b_r i \in \C \setminus \R$ for all $1 \leq j \leq l$ and $1 \leq r \leq p$.
Then we have
\[        
  \dim(\AID(\Lg)) = \dim(\Inn(\Lg)) + \sum_{j = 1, \ep_j\neq 0}^k (\varepsilon_j - 1) + 2\sum_{j = 1}^s (e_j - 1)
  + 2\sum_{j = 1}^l (f_l - 1) + 4 \sum_{j = 1}^p m_j. 
\]
\end{thm}

\begin{proof}
We may assume that $\Lg$ is canonical with $N:=k+s+l+p$ blocks. We will show the result by induction over $N$.
For $N=1$ the claim follows from the previous lemmas. If we have a canonical Lie algebra $\Lg$ with $N+1$ blocks,
we take a basis $\{ x_1,\ldots ,x_n,y_1,\ldots ,y_m,z_1,z_2\}$, where $z_1,z_2$ span $[\Lg,\Lg]$ and where
$\{x_1,\ldots ,x_n\}$ corresponds to the first $N$ blocks and $\{y_1,\ldots ,y_m\}$ to the last block. Since we have
$[x_i,y_j]=0$ for all $i,j$, we can apply Lemma $\ref{5.8}$ to show that the results holds for $N+1$ blocks if it holds for
$N$ blocks.
\end{proof}

\section*{Acknowledgments}
Dietrich Burde is supported by the Austrian Science Foun\-da\-tion FWF, grant I3248.
Karel Dekimpe and Bert Verbeke are supported by a long term structural funding, the Methusalem grant of the
Flemish Government and the Research Foundation Flanders (Project G.0F93.17N)

\end{document}